\documentclass[12pt]{amsart}
\usepackage{amscd, amssymb, graphics}
\usepackage{amsfonts}
\usepackage{amsmath}
\usepackage{amsxtra}
\usepackage{latexsym}
\usepackage[mathcal]{eucal}
\usepackage{graphics,colortbl}
\usepackage{tikz-cd}  
\input xy
\xyoption{all}
\usepackage{epsfig}
\usepackage[pdftex, bookmarks, colorlinks, breaklinks]{hyperref}

\newtheorem{theorem}{Theorem}[section]
\newtheorem*{theoremA}{Theorem A}
\newtheorem*{theoremB}{Theorem B}
\newtheorem*{theoremC}{Theorem C}
\newtheorem{corollary}[theorem]{Corollary}
\newtheorem{lemma}[theorem]{Lemma}
\newtheorem{proposition}[theorem]{Proposition}
\theoremstyle{definition}

\newtheorem{definition}[theorem]{Definition}
\numberwithin{equation}{section}

\theoremstyle{remark}
\newtheorem{remark}[theorem]{Remark}
\newtheorem{example}[theorem]{Example}

\newcommand{\ben}{\begin{enumerate}}
\newcommand{\een}{\end{enumerate}}
\newcommand{\bit}{\begin{itemize}}
\newcommand{\eit}{\end{itemize}}

\def\id{{\text{id}}}

\def\QED{\nobreak\quad\ifmmode\text{Q.E.D.}\else{\rm Q.E.D.}\fi}

\def\Aut{\operatorname{Aut}}
\def\Hom{\operatorname{Hom}}
\def\End{\operatorname{End}}

\begin{document}

\title[Shadowing Endomorphisms]{Shadowing Endomorphisms of Compact Groups}

\author[D. Peng]{Dekui Peng}
 \address[D. Peng]
{Institute of Mathematics, Nanjing Normal University, Nanjing 210024, China}
\email{pengdk10@lzu.edu.cn}



\begin{abstract}
	We characterize shadowing for continuous endomorphisms of compact Hausdorff groups. First, we prove that an endomorphism has shadowing if and only if its restriction to the identity component does, reducing the problem to compact connected groups. Passing to the stable image then reduces the analysis to surjective endomorphisms.
	
	For a compact connected abelian group $A$, let $T_{\mathbb Q}$ be the rational linear endomorphism induced by the dual map on $\widehat A\otimes_{\mathbb Z}\mathbb Q$. We show that the endomorphism of $A$ has shadowing if and only if every finite-dimensional $T_{\mathbb Q}$-invariant subspace is hyperbolic.
	
	For a general compact connected group, let $A$ and $S$ denote respectively the connected central and semisimple parts of its stable image. The induced endomorphism on $S/Z(S)$ determines an injective map on the set of simple factors. We prove that shadowing holds exactly when the rational dual map on $A$ is hyperbolic on every finite-dimensional invariant subspace and the induced map on the simple factors has no periodic point.
\end{abstract}

\maketitle

\section{Introduction}

The shadowing property is one of the basic stability notions in topological
dynamics. It asserts that every approximate orbit with sufficiently small
errors is uniformly traced by a genuine orbit. Originating in the study of
hyperbolic dynamics and structural stability, shadowing also plays an
important role in symbolic dynamics and inverse-limit representations; see
\cite{AokiHomeo, AH, Kitchens, LM, Pilyugin, Walters}. Good and Meddaugh extended the usual metric
definition to arbitrary compact Hausdorff spaces and related shadowing on
zero-dimensional compact spaces to inverse systems of shifts of finite type
\cite{GM}. Further results on inverse limits and variants of shadowing can be
found in \cite{Bernardes,CL,GOP}.

Compact group endomorphisms form a natural setting in which shadowing can be
described through algebraic structure. Compact groups carry canonical
uniformities, quotient homomorphisms lift pseudo-orbits, and extensions by
totally disconnected normal subgroups behave particularly well. In the
abelian case, Pontryagin duality converts a continuous endomorphism into an
endomorphism of a discrete torsion-free group, allowing shadowing to be
expressed in spectral and module-theoretic terms. We refer to \cite{HM,Morris}
for compact groups and Pontryagin duality, and to
\cite{KitchensSchmidt,Schmidt} for algebraic dynamics on compact groups.

The study of shadowing for compact group automorphisms goes back to Aoki. In
\cite{AokiSolenoid}, he considered automorphisms of finite-dimensional compact
connected abelian groups, or solenoids, and proved that two-sided shadowing is
equivalent to hyperbolicity of the associated finite-dimensional lifting
system. Equivalently, the corresponding rational linear automorphism has no
eigenvalue on the unit circle. Aoki also studied the relation with
topological stability; see \cite{AokiCorrection} for the subsequent
correction.

Aoki and Dateyama later introduced the OE-property and proved that, for
automorphisms of compact metrizable groups, it is equivalent to two-sided
shadowing \cite{AD}. Their structural analysis identifies the two mechanisms
that also appear in the present paper. In the connected abelian part,
finite-dimensional algebraic directions are governed by hyperbolicity, while
the nonalgebraic part is of shift type. In the semisimple part, the
automorphism permutes the simple factors of the centre-free quotient:
infinite orbits yield two-sided full shifts, whereas finite orbits produce
semisimple Lie obstructions.

The present paper gives a direct characterization for arbitrary
endomorphisms of compact groups, without assuming invertibility or
metrizability. Unlike the hyperspace formulation of the OE-property, our final
criterion is stated directly in terms of the rationalized Pontryagin dual and
the action on the simple factors of the semisimple quotient. This makes the
two sources of shadowing transparent. The passage from automorphisms to
endomorphisms also introduces genuinely new phenomena. On the abelian side,
zero may occur as an eigenvalue and does not obstruct shadowing. On the
semisimple side, the map on the simple factors is injective but need not be
surjective, so one-sided full shifts occur in addition to two-sided shifts.
Moreover, shadowing is not inherited by arbitrary closed invariant subgroups,
even for compact group automorphisms. Consequently, the reduction to the
identity component requires a separate argument.

\subsection{Main results}

Our first theorem shows that the totally disconnected quotient of a compact
group creates no additional obstruction. As usual, \(G_0\) denotes the
identity component of \(G\).

\begin{theoremA}[see Theorem~\ref{thmA}]
Let \(G\) be a compact group and let \(\alpha\in\End(G)\). Then
\(\alpha\) has shadowing if and only if \(\alpha|_{G_0}\) has
shadowing.
\end{theoremA}

The sufficiency follows from the extension theorem, since the quotient
\(G/G_0\) is totally disconnected. For the converse, we first treat
automorphisms. In the compact metrizable case, the connected structural
criterion proved later in the paper, together with two auxiliary lemmas of
Aoki and Dateyama, yields shadowing on the identity component. The general
automorphism case is then obtained from invariant metrizable quotients and the
inverse-limit theorem. Surjective endomorphisms are reduced to automorphisms
through their natural extensions, and arbitrary endomorphisms are finally
handled by passing to their stable images. Thus Theorem~A reduces the
shadowing problem for compact group endomorphisms to the connected case.

Although the proof of the metrizable automorphism case of
Lemma~\ref{lem:aut-component} is postponed until after the connected
structural criterion, there is no circularity in the argument. Indeed,
Theorem~\ref{thm:connected-endo-shad} is proved independently of
Theorem~A and Lemma~\ref{lem:aut-component}, using only the stable-image
reduction and the separate criteria for the abelian and semisimple parts.
The connected criterion is then used to complete the deferred proof of
Lemma~\ref{lem:aut-component}, and hence the proof of Theorem~A.

We next treat connected compact abelian groups. Let \(A\) be such a group,
let \(\beta\in\End(A)\), put \(\Gamma=\widehat A\), and let
\(T=\widehat\beta\). Since \(\Gamma\) is torsion-free, \(T\) extends to a
\(\mathbb Q\)-linear endomorphism \(T_{\mathbb Q}\) of
\(V=\Gamma\otimes_{\mathbb Z}\mathbb Q\). A finite-dimensional invariant
subspace is called hyperbolic if the spectrum of the restricted map is
disjoint from the unit circle. Since \(T_{\mathbb Q}\) need not be invertible,
the eigenvalue zero is allowed.

\begin{theoremB}[see Theorem~\ref{thm:abelian-endo-shad}]
Let \(A\) be a compact connected abelian group and let
\(\beta\in\End(A)\). Then \(\beta\) has shadowing if and only if
\(T_{\mathbb Q}|_W\) is hyperbolic for every finite-dimensional
\(T_{\mathbb Q}\)-invariant subspace \(W\leq V\).
\end{theoremB}

Necessity is obtained by passing to finite-dimensional solenoidal factors. For
the converse, we regard \(V\) as a \(\mathbb Q[u]\)-module. Its algebraic
part is recovered from finite-dimensional hyperbolic factors, while the
quotient by the algebraic part is torsion-free over \(\mathbb Q[u]\) and is
assembled from free modules. The duals of these free modules are one-sided
full shifts. The two parts are then recombined by the extension theorem.

For the general connected case, let
\(K=G_\infty=\bigcap_{n\geq0}\alpha^n(G)\) be the stable image and put
\(\beta=\alpha|_K\). Then \(\beta\) is surjective. Set
\(A=Z(K)_0\) and \(S=K'\). The standard structure theorem gives
\(K=AS\), where \(A\cap S\) is central and totally disconnected. Moreover,
\[
S/Z(S)\cong\prod_{i\in I}S_i,
\]
where every \(S_i\) is a centre-free simple connected compact Lie group. The
endomorphism induced by \(\beta|_S\) is encoded by an injective map
\(\tau:I\to I\), as described in Lemma~\ref{lem:semisimple-index-map}.

\begin{theoremC}[see Theorem~\ref{thm:connected-endo-shad}]
Let \(G\) be a compact connected group and let \(\alpha\in\End(G)\).
With the notation above, \(\alpha\) has shadowing if and only if:
\begin{enumerate}
\item \(T_{\mathbb Q}|_W\) is hyperbolic for every finite-dimensional
\(T_{\mathbb Q}\)-invariant subspace
\(W\leq\widehat A\otimes_{\mathbb Z}\mathbb Q\);
\item the injective map \(\tau\) has no periodic point.
\end{enumerate}
\end{theoremC}

The second condition has a direct dynamical interpretation. Since \(\tau\) is
injective, each component of its functional graph is a finite cycle, a
one-sided chain, or a two-sided chain. The absence of periodic points removes
the cyclic components, and the induced endomorphism on \(S/Z(S)\) becomes a
product of one-sided and two-sided full shifts. For automorphisms, \(\tau\) is
a permutation, so only two-sided shifts remain. Thus Theorem~C extends the
automorphism picture of Aoki and Dateyama and states it without reference to
the hyperspace or the OE-property.

The paper is organized as follows. Section~\ref{pre} recalls the definitions
and basic results concerning shadowing, inverse systems, natural extensions,
compact groups, Pontryagin duality, and the spectral notation used throughout
the paper. Section~\ref{sta} develops the reduction to the stable image.
Section~\ref{thA} proves Theorem~A, apart from the metrizable automorphism
case of Lemma~\ref{lem:aut-component}, whose proof is deferred until the
connected structural criterion has been established. Section~\ref{thBC}
proves Theorems~B and C, records the resulting automorphism criterion, and
then uses Theorem~C to complete the deferred proof. As explained above, the
proof of Theorem~C is independent of both Theorem~A and
Lemma~\ref{lem:aut-component}.

\section{Preliminaries}\label{pre}

Throughout the paper, all topological groups are Hausdorff and all homomorphisms are continuous. The identity element of a group is denoted by $e$.
 For a topological group $G$, we write $G_0$ for its identity component, $Z(G)$ for its centre, and $G'=[G,G]$ for its commutator subgroup.
  We denote by \[ \End(G) := \{\alpha:G\to G:\alpha\text{ is a continuous homomorphism}\} \] the monoid of continuous endomorphisms of $G$, 
  and by \[ \Aut(G) := \{\alpha\in\End(G):\alpha\text{ is a topological automorphism}\} \] the group of bicontinuous automorphisms of $G$.

\subsection{Shadowing}

The shadowing property was originally formulated for continuous
self-maps of compact metric spaces. Let $(X,d)$ be a compact metric
space and let $f:X\to X$ be continuous. Given $\delta>0$, a sequence
$(x_i)_{i\geq0}$ is called a $\delta$-pseudo-orbit if
\[
d\bigl(f(x_i),x_{i+1}\bigr)<\delta
\]
for every $i\geq0$. Given $\varepsilon>0$, the sequence is said to be
$\varepsilon$-shadowed by $x\in X$ if
\[
d\bigl(f^i(x),x_i\bigr)<\varepsilon
\]
for every $i\geq0$. The map $f$ has shadowing if, for every
$\varepsilon>0$, there exists $\delta>0$ such that every
$\delta$-pseudo-orbit is $\varepsilon$-shadowed by an actual orbit;
see, for example, \cite{AH,Pilyugin}.

On a compact space, any two compatible metrics are uniformly
equivalent. Consequently, shadowing does not depend on the particular
compatible metric used in its definition. It is therefore more
naturally regarded as a property of a uniform dynamical system rather
than a genuinely metric property.

Good and Meddaugh extended the definition to arbitrary compact
Hausdorff spaces using finite open covers \cite{GM}. Equivalently, one
may use the unique compatible uniformity of a compact Hausdorff space.
More precisely, if $U$ and $V$ are entourages of a uniform space
$X$, a sequence $(x_i)_{i\geq0}$ is a $V$-pseudo-orbit for
$f:X\to X$ if
\[
\bigl(f(x_i),x_{i+1}\bigr)\in V
\]
for every $i\geq0$, and it is $U$-shadowed by $x\in X$ if
\[
\bigl(f^i(x),x_i\bigr)\in U
\]
for every $i\geq0$. The resulting uniform definition agrees, on
compact spaces, with both the classical metric definition and the
open-cover definition of Good and Meddaugh.

Every topological group $G$ carries a canonical left uniformity. A
base of entourages for this uniformity is given by
\[
E_U^L
=
\{(x,y)\in G\times G:x^{-1}y\in U\},
\]
where $U$ ranges over the identity neighbourhoods of $G$. Throughout
this paper, shadowing for endomorphisms of topological groups is
understood with respect to this left uniformity.

\begin{definition}\label{def:shadowing}
Let $G$ be a topological group and let $\alpha\in\End(G)$.

A finite or infinite sequence $(x_i)$ in $G$ is called a
$V$-\emph{pseudo-orbit} for $\alpha$, where $V$ is an identity
neighbourhood, if
\[
\alpha(x_i)^{-1}x_{i+1}\in V
\]
whenever both $x_i$ and $x_{i+1}$ are defined.

Given an identity neighbourhood $U$, the sequence $(x_i)$ is said to
be $U$-\emph{shadowed} by a point $x\in G$ or, by the orbit of $x$, if
\[
\alpha^i(x)^{-1}x_i\in U
\]
for every relevant index $i$.

\begin{enumerate}
\item The endomorphism $\alpha$ has the \emph{finite shadowing
property} if, for every identity neighbourhood $U$, there exists an
identity neighbourhood $V$ such that every finite $V$-pseudo-orbit is
$U$-shadowed.

\item The endomorphism $\alpha$ has the \emph{shadowing property} if,
for every identity neighbourhood $U$, there exists an identity
neighbourhood $V$ such that every $V$-pseudo-orbit indexed by
$\mathbb N$ is $U$-shadowed.

\item If $\alpha\in\Aut(G)$, then $\alpha$ has the
\emph{two-sided shadowing property} if, for every identity
neighbourhood $U$, there exists an identity neighbourhood $V$ such
that every $V$-pseudo-orbit indexed by $\mathbb Z$ is $U$-shadowed by
a full orbit of $\alpha$.
\end{enumerate}
\end{definition}

Suppose that $G$ is metrizable, and let $d$ be any compatible
left-invariant metric on $G$. The sets
\[
\{(x,y)\in G\times G:d(x,y)<\varepsilon\},
\qquad \varepsilon>0,
\]
form a base for the left uniformity. Hence
Definition~\ref{def:shadowing} is equivalent to the usual metric
definition of shadowing with respect to $d$. In particular, the
property is independent of the choice of a compatible left-invariant
metric.

On compact spaces, finite shadowing and shadowing are equivalent. If
the map is a homeomorphism, these conditions are also equivalent to
two-sided shadowing.\footnote{These equivalences are usually stated
in the literature for compact metric spaces; see, for example,
\cite{AH,GOP,Pilyugin}. The same arguments apply, with metrics
replaced by entourages or finite open covers, to arbitrary compact
Hausdorff spaces. Indeed, compactness upgrades finite shadowing to
shadowing. For a homeomorphism, one applies one-sided shadowing to
successively longer finite portions of a two-sided pseudo-orbit and
then takes a convergent subnet of the corresponding shadowing
points.}
We shall use these equivalences without further comment.

\begin{example}\label{examples}
 We record two basic classes of group endomorphisms with shadowing. 
 \begin{enumerate} 
 \item Let $H$ be a compact metrizable group. The \emph{one-sided full shift} is the continuous surjective endomorphism 
 \[
  \sigma_+:H^{\mathbb N}\to H^{\mathbb N}, \qquad \sigma_+\bigl((x_n)_{n\geq0}\bigr) = (x_{n+1})_{n\geq0}, 
  \] 
 and the \emph{two-sided full shift} is the automorphism 
 \[ 
 \sigma:H^{\mathbb Z}\to H^{\mathbb Z}, \qquad \sigma\bigl((x_n)_{n\in\mathbb Z}\bigr) = (x_{n+1})_{n\in\mathbb Z}.
  \] 
  The map $\sigma_+$ has shadowing, while $\sigma$ has two-sided shadowing. 
  This follows from the standard diagonal tracing argument for full shifts: a sufficiently accurate pseudo-orbit is shadowed by the point obtained by reading its zeroth coordinate along the time direction. See, for example, \cite{AH,Pilyugin} for the standard shadowing argument in symbolic dynamics. 

\item Let $G$ be a compact totally disconnected group and let
$\alpha\in\End(G)$. Then $\alpha$ has shadowing.

Indeed, for every open normal subgroup $U$ of $G$, put
\[
N_U:=\bigcap_{n\geq0}\alpha^{-n}(U).
\]
Then $N_U$ is a closed normal subgroup of $G$ such that
$\alpha(N_U)\subseteq N_U$. Hence $\alpha$ induces an endomorphism
$\alpha_U$ on $G/N_U$. The orbit map
\[
\Phi_U:G/N_U\longrightarrow (G/U)^{\mathbb N},
\qquad
\Phi_U(xN_U)
=
\bigl(\alpha^n(x)U\bigr)_{n\geq0},
\]
is a topological group embedding and satisfies
\[
\Phi_U\circ\alpha_U=\sigma_+\circ\Phi_U,
\]
where $\sigma_+$ denotes the one-sided shift. Consequently,
$(G/N_U,\alpha_U)$ is conjugate to a one-sided group shift over the
finite group alphabet $G/U$.

Every one-dimensional group shift over a finite group alphabet is a
shift of finite type; see \cite{KitchensSchmidt,Schmidt}. The same finite-memory
argument applies to one-sided group shifts. It follows that
$\alpha_U$ has shadowing.

If $V\subseteq U$ are open normal subgroups, then
$N_V\subseteq N_U$, and the canonical bonding homomorphism
\[
G/N_V\longrightarrow G/N_U
\]
is surjective and intertwines $\alpha_V$ and $\alpha_U$. Moreover,
since $N_U\subseteq U$ for every $U$, we have
\[
\bigcap_U N_U=\{e\}.
\]
It follows that
\[
(G,\alpha)
\cong
\varprojlim_U (G/N_U,\alpha_U).
\]
Therefore, Theorem~\ref{thm:inverse-limit-shadowing} below implies that
$\alpha$ has shadowing.

More generally, in a separate work it is proved that every
continuous endomorphism of a totally disconnected locally compact
group has shadowing; see \cite{PengLC}.
   \end{enumerate} 
  \end{example}

\subsection{Inverse systems and natural extensions}

Let $(I,\leq)$ be a directed set. An \emph{inverse system of compact dynamical
systems} consists of systems $(X_i,f_i)$ and factor maps
$p_i^j:X_j\to X_i$ for $i\leq j$, satisfying
$p_i^i=\id_{X_i}$ and $p_i^k=p_i^j\circ p_j^k$. Its inverse limit is
\[
X=\varprojlim_{i\in I}X_i
=
\left\{(x_i)\in\prod_{i\in I}X_i:
 p_i^j(x_j)=x_i\text{ whenever }i\leq j\right\},
\]
with the coordinatewise map $f((x_i))=(f_i(x_i))$.

The inverse system satisfies the \emph{Mittag--Leffler condition} if, for every
$i\in I$, there is $j\geq i$ such that
$p_i^k(X_k)=p_i^j(X_j)$ for every $k\geq j$. In particular, every inverse
system with surjective bonding maps satisfies the Mittag--Leffler condition.
We shall use the fact that an inverse limit of shadowing systems satisfying the
Mittag--Leffler condition has shadowing; see \cite[Theorem~8]{GM} and the more
general uniform-space formulation in \cite{Bernardes}.

\begin{theorem}[Inverse-limit Shadowing Theorem]\label{thm:inverse-limit-shadowing}
 Let $(I,\leq)$ be a directed set, and let \[ \mathbf X = \bigl\{(X_i,f_i),p_i^j:i\leq j\text{ in }I\bigr\} \] be an inverse system of compact dynamical systems satisfying the Mittag--Leffler condition. 
 Suppose that $f_i$ has shadowing for every $i\in I$. Then the induced map \[ f:\varprojlim_{i\in I}X_i\longrightarrow \varprojlim_{i\in I}X_i \] has shadowing. 
 \end{theorem}

Let $f:X\to X$ be a continuous self-map of a compact space. The
\emph{natural extension} of $(X,f)$ is the inverse limit of the
inverse sequence
\[
X
\xleftarrow{\,f\,}
X
\xleftarrow{\,f\,}
X
\xleftarrow{\,f\,}
\cdots.
\]
Thus
\[
\widetilde X
=
\varprojlim(X,f)
=
\left\{
(x_0,x_1,\ldots)\in X^{\mathbb N}:
f(x_{n+1})=x_n\text{ for every }n\geq0
\right\}.
\]
It is equipped with the homeomorphism
\[
\widetilde f:\widetilde X\to\widetilde X,
\qquad
\widetilde f(x_0,x_1,\ldots)
=
\bigl(f(x_0),x_0,x_1,\ldots\bigr),
\]
whose inverse is the left shift
\[
\widetilde f^{-1}(x_0,x_1,\ldots)
=
(x_1,x_2,\ldots).
\]
The coordinate projection
\[
\pi_0:\widetilde X\to X,
\qquad
\pi_0(x_0,x_1,\ldots)=x_0,
\]
satisfies
\[
\pi_0\circ\widetilde f=f\circ\pi_0.
\]

In general, $\pi_0$ need not be surjective. Its image is precisely the
stable image
\[
X_\infty:=\bigcap_{n\geq0}f^n(X).
\]
The restriction $f|_{X_\infty}:X_\infty\to X_\infty$ is surjective,
and $\widetilde X$ may equivalently be viewed as the natural extension
of this restriction. In particular, if $f$ is surjective, then
$X_\infty=X$ and $\pi_0$ is surjective.

For compact metric spaces, Chen and Li proved that shadowing passes
from $f$ to its natural extension: if $f$ has shadowing, then
$\widetilde f$ has shadowing; see \cite[Theorem~1.3]{CL}. We shall use
the corresponding uniform version for compact spaces, and in
particular for compact groups equipped with their compatible
uniformities.

\subsection{Compact groups and Pontryagin duality}

We use standard facts about compact groups from \cite[Chapter 9]{HM}. By a theorem of Goto, $G'$ is connected and closed whenever $G$ is a connected compact group. Moreover, $G$ is called
\emph{semisimple} if $G'=G$. Since $(G')'=G'$ in connected compact groups, $G'$ is always semisimple. The Mal'cev-Levi decomposition theorem gives 
\[
G=Z(G)_0G'.
\]
Moreover, $Z(G)_0\cap G'$ is totally disconnected and central. If $S$ is seimisimple, connected and compact, then $Z(S)$ is totally disconnected and
\[
S/Z(S)\cong\prod_{i\in I}S_i,
\]
where each $S_i$ is a centre-free simple connected compact Lie group.

Let $A$ be a compact abelian group. Its Pontryagin dual is the discrete abelian
group
\[
\widehat A
=
\Hom_{\mathrm{cont}}(A,\mathbb T)
\]
of continuous homomorphisms from $A$ to $\mathbb T$, where $\mathbb T$ is the circle group. Pontryagin duality is a contravariant
equivalence between compact abelian groups and discrete abelian groups; see
\cite{Morris}. In particular, a continuous surjective homomorphism of compact
abelian groups dualizes to an injective homomorphism of discrete groups, and a
short exact sequence of compact abelian groups dualizes to a short exact
sequence in the opposite direction. The group $A$ is connected if and only if
$\widehat A$ is torsion-free, and $A$ is finite-dimensional if and only if
$\widehat A$ has finite free rank.

If $\beta:A\to A$ is an endomorphism, its dual endomorphism is
$T=\widehat\beta:\widehat A\to\widehat A$, defined by
$T(\chi)=\chi\circ\beta$. When $A$ is connected, put
\[
V=\widehat A\otimes_{\mathbb Z}\mathbb Q.
\]
The tensor product is the \emph{rationalization} (also known as the \emph{divisible hull}) of the torsion-free group
$\widehat A$, and
$\dim_{\mathbb Q}V=\operatorname{rank}(\widehat A)$. The map $T$ extends
uniquely to the $\mathbb Q$-linear endomorphism
$T_{\mathbb Q}=T\otimes\id_{\mathbb Q}$ of $V$.

We regard $V$ as a module over the polynomial ring $\mathbb Q[u]$ by setting
$u\cdot v=T_{\mathbb Q}v$. Its algebraic part is
\[
V_{\mathrm{alg}}
=
\{v\in V:p(T_{\mathbb Q})v=0
\text{ for some }0\neq p\in\mathbb Q[u]\}.
\]
For automorphisms one may equivalently work over the Laurent polynomial ring
$\mathbb Q[u,u^{-1}]$. This module-theoretic viewpoint is standard in
algebraic dynamics; see \cite{KitchensSchmidt,Schmidt}.

\subsection{Spectrum and hyperbolicity}

Let $W$ be a finite-dimensional vector space over $\mathbb Q$, and let
$L:W\to W$ be linear. We write $\operatorname{Spec}(L)$ for the spectrum of
the complexification $L_{\mathbb C}$. The map $L$ is called
\emph{hyperbolic} if
\[
\operatorname{Spec}(L)\cap\mathbb S^1=\varnothing,
\qquad
\mathbb S^1=\{z\in\mathbb C:|z|=1\}.
\]
For an endomorphism, zero is allowed as an eigenvalue. When $L$ is invertible,
this is the usual hyperbolicity condition for a linear automorphism.

\section{Stable images and shadowing}\label{sta}

Let $G$ be a compact group and let $\alpha:G\to G$ be an endomorphism.
Define the \emph{stable image} of $\alpha$ by
\[
G_\infty:=\bigcap_{m\geq 0}\alpha^m(G).
\]

\begin{lemma}\label{lem:stable-image-surjective}
The subgroup $G_\infty$ is compact and $\alpha$-invariant, and
$
\alpha(G_\infty)=G_\infty.
$
\end{lemma}

\begin{proof}
It is clear that $G_\infty$ is a compact subgroup and that
$
\alpha(G_\infty)\subseteq G_\infty.
$

Let $x\in G_\infty$. For each $m\geq 0$, put
$
C_m:=\{y\in\alpha^m(G):\alpha(y)=x\}.
$
Since $x\in\alpha^{m+1}(G)$, the set $C_m$ is nonempty. Moreover, the
sets $C_m$ are compact and decreasing. Hence
$
\bigcap_{m\geq 0}C_m\neq\varnothing.
$
If $y$ belongs to this intersection, then $y\in G_\infty$ and
$\alpha(y)=x$. Thus $\alpha(G_\infty)=G_\infty$.
\end{proof}

We shall also need the relation between the stable image and the
identity component.

\begin{lemma}\label{lem:stable-component}
Let $G_\infty$ be the stable image of $\alpha$. Then
\[
(G_\infty)_0
=
\bigcap_{m\geq 0}\alpha^m(G_0).
\]
In particular, the right-hand side is the stable image of
$\alpha|_{G_0}$.
\end{lemma}

\begin{proof}
We first recall that if $\varphi:G\to H$ is a continuous surjective
homomorphism of compact groups, then
\[
\varphi(G_0)=H_0.
\]
Indeed, $\varphi(G_0)$ is connected, so
$
\varphi(G_0)\subseteq H_0.
$
On the other hand, the quotient $H/\varphi(G_0)$ is a continuous
homomorphic image of the totally disconnected compact group $G/G_0$.
It is therefore totally disconnected. The image of $H_0$ in this
quotient is connected, and hence trivial. Thus
$
H_0\subseteq\varphi(G_0).
$

Applying this observation to the surjective homomorphism
$
\alpha^m:G\to\alpha^m(G)
$
gives
\[
\bigl(\alpha^m(G)\bigr)_0=\alpha^m(G_0).
\]

Put
$
L:=\bigcap_{m\geq 0}\alpha^m(G_0).
$
The groups $\alpha^m(G_0)$ form a decreasing family of nonempty compact
connected groups. Hence $L$ is connected. Since $L\subseteq G_\infty$,
we have
$
L\subseteq(G_\infty)_0.
$

Conversely,
$
(G_\infty)_0
\subseteq\bigl(\alpha^m(G)\bigr)_0
=\alpha^m(G_0)
$
for every $m$. Hence
$
(G_\infty)_0\subseteq L.
$
This proves the equality.
\end{proof}

\begin{lemma}\label{lem:stable-image}
	Let $G$ be a compact group and let $\alpha\in\End(G)$. Put
	$\alpha_\infty:=\alpha|_{G_\infty}$. Then $\alpha$ has finite
	shadowing if and only if $\alpha_\infty$ has finite shadowing.
\end{lemma}

\begin{proof}
	We first prove the necessity. Let $U$ be an identity neighbourhood in
	$G_\infty$.  
	Choose an open identity neighbourhood $U_0$ in $G$ such that
	$\overline{U_0}\cap G_\infty\subseteq U$, and let $V_0$ witness finite
	$U_0$-shadowing for $\alpha$. Put $V:=V_0\cap G_\infty$.
	
	Let $x_0,x_1,\ldots,x_n$ be a finite $V$-pseudo-orbit in
	$G_\infty$. For every $m\geq1$, the surjectivity of $\alpha_\infty$
	allows us to choose
	\[
	x_{-m},x_{-m+1},\ldots,x_{-1}\in G_\infty
	\]
	such that $\alpha(x_i)=x_{i+1}$ for $-m\leq i<0$. Thus
	\[
	x_{-m},x_{-m+1},\ldots,x_{-1},x_0,\ldots,x_n
	\]
	is a finite $V_0$-pseudo-orbit in $G$. Choose $y_m\in G$ which
	$U_0$-shadows this pseudo-orbit, and put $z_m:=\alpha^m(y_m)$. Then
	$z_m\in\alpha^m(G)$ and
	$\alpha^i(z_m)^{-1}x_i\in U_0$
	for every $0\leq i\leq n$.
	
	For $m\geq1$, define
	\[
	F_m
	:=
	\left\{
	z\in\alpha^m(G):
	\alpha^i(z)^{-1}x_i\in\overline{U_0}
	\text{ for every }0\leq i\leq n
	\right\}.
	\]
	Each $F_m$ is nonempty and compact, and the family $(F_m)_{m\geq1}$
	is decreasing. Hence
	$\bigcap_{m\geq1}F_m\neq\varnothing.$
	Choose $z$ in this intersection. Then $z\in G_\infty$, and
	$
	\alpha^i(z)^{-1}x_i\in\overline{U_0}\subseteq C
	$
	for every $0\leq i\leq n$. Since both $\alpha^i(z)$ and $x_i$
	belong to $G_\infty$, we obtain
	\[
	\alpha^i(z)^{-1}x_i\in C\cap G_\infty\subseteq U.
	\]
	Thus $\alpha_\infty$ has finite shadowing.
	
	Conversely, suppose that $\alpha_\infty$ has finite shadowing. Let
	$U$ be an identity neighbourhood in $G$. Choose an open identity
	neighbourhood $C$ such that $CC\subseteq U$. There is an identity
	neighbourhood $D_\infty$ in $G_\infty$ such that every finite
	$D_\infty$-pseudo-orbit of $\alpha_\infty$ is
	$(C\cap G_\infty)$-shadowed.
	
	Choose an open identity neighbourhood $D$ in $G$ such that
	\[
	D\subseteq C
	\qquad\text{and}\qquad
	D\cap G_\infty\subseteq D_\infty.
	\]
	By continuity, choose a symmetric identity neighbourhood $B$ in $G$
	such that
	\[
	\alpha(B)BB\subseteq D.
	\]
	In particular, $B\subseteq D$ and $BB\subseteq D$.
	
	For $m\geq0$, put $G^{(m)}:=\alpha^m(G)$. We claim that there is
	$m\geq1$ such that
	\[
	G^{(m)}\subseteq G_\infty B.
	\]
	Otherwise, every compact set $G^{(m)}\setminus G_\infty B$	would be nonempty. These sets form a decreasing family, so their
	intersection would be nonempty. However,
	\[
	\bigcap_{m\geq0}
	\left(G_m\setminus G_\infty B\right)
	=
	G_\infty\setminus G_\infty B
	=
	\varnothing,
	\]
	a contradiction.
	
	Fix such an $m$. Choose an identity neighbourhood $V$ in $G$ such
	that
	\[
	\alpha^m(V)\alpha^{m-1}(V)\cdots\alpha(V)V\subseteq B.
	\]
	
	Let $x_0,x_1,\ldots,x_n$ be a finite $V$-pseudo-orbit for $\alpha$. Set
\[
e_i:=\alpha(x_i)^{-1}x_{i+1}\in V
\qquad
\text{and}
\qquad
d_i:=\alpha^i(x_0)^{-1}x_i
\qquad(0\leq i\leq n).
\]
Then $d_0=e$. Moreover, since
$x_{i+1}=\alpha(x_i)e_i$, we have
\[
\begin{aligned}
	d_{i+1}
	&=\alpha^{i+1}(x_0)^{-1}x_{i+1}\\
	&=\alpha^{i+1}(x_0)^{-1}\alpha(x_i)e_i\\
	&=\alpha\bigl(\alpha^i(x_0)^{-1}x_i\bigr)e_i\\
	&=\alpha(d_i)e_i.
\end{aligned}
\]
In particular,
\[
d_1=e_0,\qquad
d_2=\alpha(e_0)e_1,\qquad
d_3=\alpha^2(e_0)\alpha(e_1)e_2.
\]
Therefore, by induction, for every $1\leq i\leq n$,
\[
d_i
=
\alpha^{i-1}(e_0)\alpha^{i-2}(e_1)\cdots
\alpha(e_{i-2})e_{i-1}.
\]	
	
	For $0\leq i\leq m$, put $a_i=e$ and $b_i=d_i$. For $i>m$, put
	\[
	a_i
	:=
	\alpha^{i-1}(e_0)\cdots\alpha^m(e_{i-m-1})
	\qquad
	\text{and}
	\qquad
	b_i
	:=
	\alpha^{m-1}(e_{i-m})\cdots\alpha(e_{i-2})e_{i-1}.
	\]
	Then $d_i=a_ib_i$, with $a_i\in G^{(m)}$ and $b_i\in B$ for every $i>m$.
	
	For $0\leq i\leq m$, put $k_i=q_i=e$. If $i>m$, choose
	$k_i\in G_\infty$ and $q_i\in B$ such that $a_i=k_iq_i$,
	which is possible because $a_i\in G^{(m)}\subseteq G_\infty B$.
	
	We claim that $k_0,k_1,\ldots,k_n$ is a finite
	$D_\infty$-pseudo-orbit for $\alpha_\infty$. This is immediate for
	$i<m$. If $m\leq i<n$, then $a_{i+1}=\alpha(a_i)\alpha^m(e_{i-m})$.
	Using $a_i=k_iq_i$ and $a_{i+1}=k_{i+1}q_{i+1}$, we obtain
	\[
	\alpha(k_i)^{-1}k_{i+1}
	=
	\alpha(q_i)\alpha^m(e_{i-m})q_{i+1}^{-1}.
	\]
	The right-hand side belongs to $\alpha(B)BB\subseteq D$. Since the
	left-hand side belongs to $G_\infty$, it follows that
	\[
	\alpha(k_i)^{-1}k_{i+1}
	\in D\cap G_\infty
	\subseteq D_\infty.
	\]
	This proves the claim.
	
	Choose $c\in G_\infty$ which $(C\cap G_\infty)$-shadows this
	pseudo-orbit, so that
	\[
	\alpha^i(c)^{-1}k_i\in C\cap G_\infty
	\]
	for every $0\leq i\leq n$. Put $y:=x_0c$.
	
	If $i<m$, then $k_i=e$ and $d_i=b_i\in B$. Hence
	\[
	\alpha^i(y)^{-1}x_i
	=
	\alpha^i(c)^{-1}d_i
	\in CB
	\subseteq CD
	\subseteq CC
	\subseteq U.
	\]
	If $i\geq m$, then $d_i=k_iq_ib_i$, where $q_i,b_i\in B$. Therefore
	\[
	\alpha^i(y)^{-1}x_i
	=
	\alpha^i(c)^{-1}k_iq_ib_i
	\in CBB
	\subseteq CD
	\subseteq CC
	\subseteq U.
	\]
	Thus $y$ $U$-shadows the original pseudo-orbit, and $\alpha$ has
	finite shadowing.
\end{proof}

\section{Shadowing and the identity component of a compact group}\label{thA}

In this section we prove Theorem~A. Throughout the section, we use
without further mention the equivalence between shadowing and finite
shadowing on compact spaces. We first establish the permanence of
finite shadowing under quotients and extensions, and then treat
automorphisms, surjective endomorphisms, and arbitrary endomorphisms
in turn.

We emphasize that shadowing is
not, in general, inherited by closed invariant subgroups, even for
automorphisms of compact connected abelian groups. Indeed, let
\[
\sigma:\mathbb T^{\mathbb Z}\longrightarrow\mathbb T^{\mathbb Z}
\]
be the two-sided full shift,
$\sigma((x_n)_{n\in\mathbb Z})=(x_{n+1})_{n\in\mathbb Z}$. Then
$\sigma$ has two-sided shadowing. However, the diagonal subgroup
\[
D:=\{(x)_{n\in\mathbb Z}:x\in\mathbb T\}
\]
is closed and $\sigma$-invariant, while
$\sigma|_D=\id_D$. Since $D\cong\mathbb T$, the restriction
$\sigma|_D$ does not have shadowing. Thus the passage from an
automorphism of a compact group to its identity component requires
the special structure of the latter and cannot be deduced from a
general hereditary property.

We establish the permanence of
finite shadowing under quotients and extensions, and then treat automorphisms,
surjective endomorphisms, and arbitrary endomorphisms in turn.

\subsection{Quotients and extensions}

We first record two general permanence properties.

\begin{lemma}\label{lem:quotient-shadowing}
Let $G$ and $H$ be topological groups, let $\alpha:G\to G$ and
$\beta:H\to H$ be endomorphisms, and let $\pi:G\to H$ be an open continuous surjective homomorphism satisfying
$\pi\circ\alpha=\beta\circ\pi$. If $\alpha$ has finite shadowing, then
$\beta$ has finite shadowing.
\end{lemma}

\begin{proof}
Let $U$ be an identity neighbourhood in $H$, and put
$U_0:=\pi^{-1}(U)$. Let $V_0$ be an identity neighbourhood in $G$
witnessing finite $U_0$-shadowing for $\alpha$. Since $\pi$ is open,
$\pi(V_0)$ is an identity neighbourhood in $H$.

Let $y_0,y_1,\ldots,y_n$ be a finite $\pi(V_0)$-pseudo-orbit for
$\beta$. We construct a finite $V_0$-pseudo-orbit
$x_0,x_1,\ldots,x_n$ for $\alpha$ such that $\pi(x_i)=y_i$ for every
$i$.

Choose $x_0\in\pi^{-1}(y_0)$. Suppose that $x_i$ has been chosen. Since
$\beta(y_i)^{-1}y_{i+1}\in\pi(V_0)$, choose $v_i\in V_0$ such that
$\pi(v_i)=\beta(y_i)^{-1}y_{i+1}$, and put
$x_{i+1}:=\alpha(x_i)v_i$. Then
\[
\pi(x_{i+1})
=\beta(\pi(x_i))\pi(v_i)
=\beta(y_i)\beta(y_i)^{-1}y_{i+1}
=y_{i+1},
\]
and $\alpha(x_i)^{-1}x_{i+1}=v_i\in V_0$. Thus
$x_0,x_1,\ldots,x_n$ is a finite $V_0$-pseudo-orbit.

Let $x\in G$ $U_0$-shadow this pseudo-orbit. Then, for every
$0\leq i\leq n$,
\[
\beta^i(\pi(x))^{-1}y_i
=\pi\bigl(\alpha^i(x)^{-1}x_i\bigr)
\in\pi(U_0)
=U.
\]
Hence $\pi(x)$ $U$-shadows the original pseudo-orbit.
\end{proof}

\begin{lemma}\label{lem:extension-shadowing}
Let $G$ be a compact group, let $\alpha:G\to G$ be an endomorphism,
and let $N\trianglelefteq G$ be a closed normal subgroup such that
$\alpha(N)\subseteq N$. Let $\alpha_N:=\alpha|_N$, and let
$\bar\alpha:G/N\to G/N$ be the induced endomorphism. If both
$\alpha_N$ and $\bar\alpha$ have finite shadowing, then $\alpha$ has
finite shadowing.
\end{lemma}

\begin{proof}
Let $\pi:G\to G/N$ be the quotient map. Given an identity
neighbourhood $U$ in $G$, choose a symmetric identity neighbourhood
$U'$ such that $U'U'\subseteq U$.

Choose a symmetric identity neighbourhood $W'$ in $G$ such that
$W'\cap N$ witnesses finite $(U'\cap N)$-shadowing for $\alpha_N$.
By the continuity of $\alpha$ and multiplication, there is a symmetric
identity neighbourhood $W$ in $G$ such that
$\alpha(W)WW\subseteq W'$. Shrinking $W$ if necessary, we may also
assume that $W\subseteq U'$.

By finite shadowing for $\bar\alpha$, choose an identity neighbourhood
$\bar V$ in $G/N$ such that every finite $\bar V$-pseudo-orbit is
$\pi(W)$-shadowed. Choose a symmetric identity neighbourhood $V$ in
$G$ such that $V\subseteq W$ and $\pi(V)\subseteq\bar V$.

Let $x_0,x_1,\ldots,x_n$ be a finite $V$-pseudo-orbit for $\alpha$.
Put $e_i:=\alpha(x_i)^{-1}x_{i+1}\in V$ for $0\leq i<n$. Then
\[
\bar\alpha(\pi(x_i))^{-1}\pi(x_{i+1})
=\pi(e_i)\in\pi(V)\subseteq\bar V.
\]
Thus $\pi(x_0),\pi(x_1),\ldots,\pi(x_n)$ is a finite
$\bar V$-pseudo-orbit for $\bar\alpha$.

Choose $q\in G/N$ which $\pi(W)$-shadows this pseudo-orbit, and choose
$x\in G$ such that $\pi(x)=q$. For $0\leq i\leq n$, put
$d_i:=\alpha^i(x)^{-1}x_i$. Then
\[
\pi(d_i)
=\bar\alpha^i(q)^{-1}\pi(x_i)\in\pi(W),
\]
and hence $d_i\in\pi^{-1}(\pi(W))=WN=NW$. We may therefore choose
$y_i\in N$ and $w_i\in W$ such that $d_i=y_iw_i$.

Since $d_{i+1}=\alpha(d_i)e_i$, we have
\[
\alpha(y_i)\alpha(w_i)e_i=y_{i+1}w_{i+1}.
\]
It follows that
\[
\alpha(y_i)^{-1}y_{i+1}
=\alpha(w_i)e_iw_{i+1}^{-1}
\in\alpha(W)VW
\subseteq\alpha(W)WW
\subseteq W'.
\]
Since the left-hand side belongs to $N$, the sequence
$y_0,y_1,\ldots,y_n$ is a finite $(W'\cap N)$-pseudo-orbit for
$\alpha_N$.

Consequently, there is $y\in N$ such that
$\alpha^i(y)^{-1}y_i\in U'\cap N$ for every $0\leq i\leq n$. Put
$z:=xy$. Then
\[
\alpha^i(z)^{-1}x_i
=\alpha^i(y)^{-1}d_i
=\alpha^i(y)^{-1}y_iw_i
\in U'W
\subseteq U'U'
\subseteq U
\]
for every $0\leq i\leq n$. Therefore $z$ $U$-shadows
$x_0,x_1,\ldots,x_n$.
\end{proof}

As an immediate consequence, one obtains one direction of the desired
identity-component reduction.

\begin{proposition}\label{prop:component-to-group}
Let $G$ be a compact group and let $\alpha:G\to G$ be an endomorphism.
If $\alpha|_{G_0}$ has finite shadowing, then $\alpha$ has finite
shadowing.
\end{proposition}

\begin{proof}
It is evident that
$
\alpha(G_0)\subseteq G_0.
$

The quotient $G/G_0$ is compact and totally disconnected. Therefore
the induced endomorphism on $G/G_0$ has finite shadowing by Example \ref{examples}(2). The conclusion
now follows from Lemma~\ref{lem:extension-shadowing}.
\end{proof}

By the preceding proposition, it remains only to prove the necessity in Theorem~A. For this purpose, we need the following lemma for \textbf{automorphisms}. We first deduce its general compact case from the metrizable case, while the proof of the latter will be postponed until the end of the paper.

\begin{lemma}\label{lem:aut-component}
	Let $G$ be a compact group and let
	$\gamma\in\operatorname{Aut}(G)$. If $\gamma$ has shadowing, then
	$\gamma|_{G_0}$ has shadowing.
\end{lemma}

\begin{proof}[Proof assuming the metrizable case]
	Let $\mathcal N$ be the family of all closed normal subgroups $N$ of
	$G$ such that
	\[
	\gamma(N)=N
	\qquad\text{and}\qquad
	G/N\text{ is metrizable}.
	\]
	We first observe that
	\[
	\bigcap_{N\in\mathcal N}N=\{e\}.
	\]
	Indeed, let $U$ be an identity neighbourhood in $G$. By the standard
	structure theory of compact groups, there is a closed normal subgroup
	$K\leq G$ such that
	\[
	K\subseteq U
	\qquad\text{and}\qquad
	G/K\text{ is a compact Lie group}.
	\]
	Put
	\[
	N_U:=\bigcap_{n\in\mathbb Z}\gamma^n(K).
	\]
	Then $N_U$ is closed and normal, $\gamma(N_U)=N_U$, and
	$N_U\subseteq K\subseteq U$. Moreover, the natural homomorphism
	\[
	G/N_U\longrightarrow
	\prod_{n\in\mathbb Z}G/\gamma^n(K)
	\]
	is a topological embedding. Since every factor on the right is a
	compact Lie group, the product is metrizable. Hence $G/N_U$ is
	metrizable, so $N_U\in\mathcal N$. This proves the claim.
	
	The family $\mathcal N$ is directed under reverse inclusion. Indeed,
	if $N_1,N_2\in\mathcal N$, then $N_1\cap N_2$ is closed, normal, and
	$\gamma$-invariant. Moreover, the canonical homomorphism
	\[
	G/(N_1\cap N_2)
	\longrightarrow
	G/N_1\times G/N_2
	\]
	is a topological embedding. Hence $G/(N_1\cap N_2)$ is metrizable,
	and therefore $N_1\cap N_2\in\mathcal N$.
	
	For each $N\in\mathcal N$, let $\gamma_N:G/N\longrightarrow G/N$
	be the automorphism induced by $\gamma$.  Whenever $M\subseteq N$, let
	\[
	\varphi_N^M:G/M\longrightarrow G/N,
	\qquad
	\varphi_N^M(xM)=xN,
	\]
	be the canonical quotient homomorphism. These maps are surjective and
	intertwine the corresponding automorphisms. Thus
	\[
	\bigl\{(G/N,\gamma_N),\varphi_N^M:M\subseteq N\bigr\}
	\]
	is an inverse system of compact metrizable dynamical systems, and the
	canonical map gives a conjugacy
	\[
	(G,\gamma)
	\cong
	\varprojlim_{N\in\mathcal N}(G/N,\gamma_N).
	\]
	
	Taking identity components, we obtain
	\[
	(G_0,\gamma|_{G_0})
	\cong
	\varprojlim_{N\in\mathcal N}
	\bigl((G/N)_0,\gamma_N|_{(G/N)_0}\bigr).
	\]
	To justify this, recall that a surjective homomorphism between compact
	groups maps the identity component onto the identity component.
	Consequently, all the restricted bonding maps
	\[
	(\varphi_N^M)|_{(G/M)_0}:(G/M)_0\longrightarrow(G/N)_0
	\]
	are surjective. Their inverse limit is therefore connected and is
	contained in the identity component of
	$\varprojlim_{N\in\mathcal N}G/N$. The reverse inclusion follows
	because every coordinate projection maps the identity component into
	the identity component.
	
	By Lemma~\ref{lem:quotient-shadowing}, each $\gamma_N$ has shadowing.
	Since $G/N$ is metrizable, the metrizable case of the present lemma
	implies that
	\[
	\gamma_N|_{(G/N)_0}
	\]
	has shadowing for every $N\in\mathcal N$. Theorem~
	\ref{thm:inverse-limit-shadowing} now shows that
	$\gamma|_{G_0}$ has shadowing.
\end{proof}
\subsection{Natural extensions of surjective endomorphisms}
In this subsection, we let $K$ be a compact group and let $\beta:K\to K$ be a surjective
endomorphism. 

\begin{lemma}\label{lem:natural-extension-shadowing}
	If $\beta$ has finite shadowing, then $\widetilde\beta$ has finite
	shadowing.
\end{lemma}

\begin{proof}
	Since $K$ is compact, finite shadowing is equivalent to shadowing.
	The natural extension is the inverse limit of the inverse sequence
	\[
	(K,\beta)
	\xleftarrow{\ \beta\ }
	(K,\beta)
	\xleftarrow{\ \beta\ }
	(K,\beta)
	\xleftarrow{\ \beta\ }
	\cdots.
	\]
	All bonding maps are surjective, so the inverse system satisfies the
	Mittag--Leffler condition. Moreover, the coordinatewise map induced
	on the inverse limit is precisely $\widetilde\beta$, since for
	$(x_j)_{j\geq0}\in\widetilde K$,
	\[
	\bigl(\beta(x_0),\beta(x_1),\beta(x_2),\ldots\bigr)
	=
	\bigl(\beta(x_0),x_0,x_1,\ldots\bigr).
	\]
	Theorem~\ref{thm:inverse-limit-shadowing} therefore implies that
	$\widetilde\beta$ has shadowing, and hence finite shadowing.
\end{proof}

\begin{lemma}\label{lem:natural-extension-component}
Let $\beta:K\to K$ be a surjective endomorphism. Then
\begin{equation*}
(\widetilde K)_0
=
\left\{
(x_j)_{j\geq 0}\in\widetilde K:
x_j\in K_0 \text{ for every } j\geq 0
\right\}.
\end{equation*}
Equivalently,
$(\widetilde K)_0=\varprojlim(K_0,\beta|_{K_0})$.
Moreover, the zeroth-coordinate projection
$\pi_0:(\widetilde K)_0\to K_0$ is surjective.
\end{lemma}

\begin{proof}
Since $\beta$ is surjective, $\beta(K_0)=K_0$. Indeed, a continuous
surjective homomorphism between compact groups maps the identity
component onto the identity component.

Put
\begin{equation*}
L
=
\left\{
(x_j)_{j\geq 0}\in\widetilde K:
x_j\in K_0 \text{ for every } j\geq 0
\right\}.
\end{equation*}
The group $L$ is the inverse limit of compact connected groups with
surjective bonding maps. Hence $L$ is connected, and therefore
$L\subseteq(\widetilde K)_0$.

Conversely, for every $j\geq 0$, the coordinate projection
$\pi_j:\widetilde K\to K$ is continuous. Hence
$\pi_j((\widetilde K)_0)\subseteq K_0$, and consequently
$(\widetilde K)_0\subseteq L$. Thus $(\widetilde K)_0=L$.

Finally, the bonding map
$\beta|_{K_0}:K_0\to K_0$ is surjective. Hence every coordinate
projection from the inverse limit $\varprojlim(K_0,\beta|_{K_0})$
onto $K_0$ is surjective. In particular,
$\pi_0:(\widetilde K)_0\to K_0$ is surjective.
\end{proof}

We can now prove the remaining implication for surjective
endomorphisms.

\begin{corollary}\label{surshad}
Let $G$ be a compact group and let $\alpha:G\to G$ be a surjective
endomorphism satisfying the finite shadowing property. Then $\alpha|_{G_0}$ has finite shadowing as well.
\end{corollary}

\begin{proof}
Suppose that $\alpha$ has finite shadowing. Let
$
(\widetilde G,\widetilde\alpha)
$
be its natural extension. By Lemma~\ref{lem:natural-extension-shadowing},
$\widetilde\alpha$ has finite shadowing and hence shadowing. Since $\widetilde\alpha$ is an
automorphism, Lemma~\ref{lem:aut-component} implies that
$
\widetilde\alpha|_{(\widetilde G)_0}
$
has finite shadowing.

By Lemma~\ref{lem:natural-extension-component}, the zeroth-coordinate
projection
$
p_0:(\widetilde G)_0\to G_0
$
is a continuous surjective homomorphism satisfying
\[
p_0\circ\widetilde\alpha
=
\alpha|_{G_0}\circ p_0.
\]
Lemma~\ref{lem:quotient-shadowing} now shows that
$\alpha|_{G_0}$ has finite shadowing.
\end{proof}

\subsection{Proof of Theorem A}

For an arbitrary compact group endomorphism, the preceding arguments
give the following rigorous reduction.

\begin{proposition}\label{prop:general-reduction}
Let $G$ be a compact group and let $\alpha:G\to G$ have finite
shadowing. 
Then the restriction
$
\alpha\big|_{\bigcap_{n\geq 0}\alpha^n(G_0)}
$
has finite shadowing.
\end{proposition}

\begin{proof}
	Put $K:=G_\infty$. By Lemma~\ref{lem:stable-image},
	$\alpha|_K$ has finite shadowing, and it is surjective by
	Lemma~\ref{lem:stable-image-surjective}. Hence
Corollary~\ref{surshad} implies that
$\alpha|_{K_0}$ has finite shadowing. Finally,
Lemma~\ref{lem:stable-component} gives
$
K_0=\bigcap_{n\geq 0}\alpha^n(G_0).
$
\end{proof}

We can now complete the proof of Theorem~A.

\begin{theorem}\label{thmA}
 Let $G$ be a compact group and let
$\alpha\in \End(G)$. Then
\[
\alpha\text{ has shadowing}
\quad\Longleftrightarrow\quad
\alpha|_{G_0}\text{ has shadowing}.
\]
\end{theorem}

\begin{proof}
By the equivalence between finite shadowing and shadowing, it suffices
to work with finite shadowing.

If $\alpha|_{G_0}$ has finite shadowing, then $\alpha$ has finite
shadowing by Proposition~\ref{prop:component-to-group}.

Conversely, suppose that $\alpha$ has finite shadowing. By
Proposition~\ref{prop:general-reduction}, the restriction of $\alpha$
to
$
\bigcap_{n\geq 0}\alpha^n(G_0)
$
has finite shadowing. This group is the stable image of
$\alpha|_{G_0}$. Lemma~\ref{lem:stable-image} therefore implies
that $\alpha|_{G_0}$ has finite shadowing.
\end{proof}

\section{Endomorphisms of connected compact groups}\label{thBC}

Throughout this section, $G$ denotes a connected compact
group and $\alpha:G\to G$ is a continuous endomorphism, except in the final subsection. Recall that again, we use the
equivalence between shadowing and finite shadowing on compact spaces
without further mention.

The purpose of this section is to characterize shadowing directly in
terms of the algebraic structure of $\alpha$. The abelian part is
described by the action induced on the Pontryagin dual, while the
semisimple part is described by an injective map on the index set of
the simple Lie factors.

By Lemma~\ref{lem:stable-image}, we have
\begin{equation}\label{eq:stable-image-equiv}
\alpha\text{ has shadowing}
\quad\Longleftrightarrow\quad
\alpha_\infty\text{ has shadowing}.
\end{equation}

Since every $\alpha^n(G)$ is compact and connected and the sequence
$(\alpha^n(G))_{n\geq0}$ is decreasing, the stable image $G_\infty$
is connected.
Thus the shadowing problem for an arbitrary endomorphism reduces to
the corresponding problem for a surjective endomorphism, since $\alpha_\infty$ is surjective. This
reduction is particularly useful because the connected central and
semisimple parts of $G_\infty$ are invariant under
$\alpha_\infty$.

\subsection{The connected abelian case}
Throughout this subsection, let $A$ be a connected compact abelian group and let
$\beta:A\to A$ be an endomorphism. Put
\[
\Gamma:=\widehat A
\qquad\text{and}\qquad
T:=\widehat\beta:\Gamma\to\Gamma.
\]
Since $A$ is connected, $\Gamma$ is torsion-free. Let
\[
V:=\Gamma\otimes_{\mathbb Z}\mathbb Q
\]
and denote by $T_{\mathbb Q}:V\to V$ the $\mathbb Q$-linear map induced by $T$.

\begin{proposition}\label{prop:finite-dim}
If $A$ is finite-dimensional, or equivalently, if $\Gamma$ is of
finite rank, then $\beta$ has shadowing if and only if
$T_{\mathbb Q}$ is hyperbolic.
\end{proposition}

\begin{proof}
Let $n=\dim A=\operatorname{rank}(\Gamma)$. Identifying
$V:=\Gamma\otimes_{\mathbb Z}\mathbb Q$ with $\mathbb Q^n$, we regard
$\Gamma$ as a subgroup of $\mathbb Q^n$. The inclusion
$\Gamma\hookrightarrow\mathbb Q^n$ dualizes to a continuous
surjective homomorphism
\[
\pi:\widehat{\mathbb Q}^{\,n}\longrightarrow A.
\]
Its kernel $N$ is Pontryagin dual to $\mathbb Q^n/\Gamma$. Since
$\Gamma$ has rank $n$, the group $\mathbb Q^n/\Gamma$ is torsion.
Consequently, $N$ is totally disconnected.

The homomorphism $T:\Gamma\to\Gamma$ extends uniquely to the rational
linear endomorphism $T_{\mathbb Q}:\mathbb Q^n\to\mathbb Q^n$. Let
$\widetilde\beta:\widehat{\mathbb Q}^{\,n}\to
\widehat{\mathbb Q}^{\,n}$ be its Pontryagin dual. Since
$T_{\mathbb Q}|_\Gamma=T$, we have
$\pi\circ\widetilde\beta=\beta\circ\pi$. In particular,
$\widetilde\beta(N)\subseteq N$.

The restriction $\widetilde\beta|_N$ has shadowing, since $N$ is
compact and totally disconnected. It therefore follows from
Lemma~\ref{lem:extension-shadowing} and
Lemma~\ref{lem:quotient-shadowing} that
\[
\beta\text{ has shadowing}
\quad\Longleftrightarrow\quad
\widetilde\beta\text{ has shadowing}.
\]

We now apply the Fitting decomposition to $T_{\mathbb Q}$.

For all sufficiently large $m$, the Fitting decomposition gives
\[
\mathbb Q^n
=
\ker T_{\mathbb Q}^m
\oplus
\operatorname{im}T_{\mathbb Q}^m.
\]
Fix such an $m$ and put 
\[
\mathbb Q^n
=
V_0\oplus V_1,
\qquad
V_0:=\ker T_{\mathbb Q}^m,
\qquad
V_1:=\operatorname{im}T_{\mathbb Q}^m.
\]
Both $V_0$ and $V_1$ are $T_{\mathbb Q}$-invariant,
$T_{\mathbb Q}|_{V_0}$ is nilpotent, and
$T_{\mathbb Q}|_{V_1}$ is an automorphism. Accordingly,
\[
\widehat{\mathbb Q}^{\,n}
\cong
\widehat{V_0}\times\widehat{V_1},
\]
and, under this identification,
$\widetilde\beta=\widetilde\beta_0\times\widetilde\beta_1$, where
$\widetilde\beta_i$ is dual to $T_{\mathbb Q}|_{V_i}$.

The endomorphism $\widetilde\beta_0$ has shadowing. Indeed, it is
nilpotent. More generally, if $\varphi^m(x)=e$ for every $x$, then $\varphi$ has finite shadowing. To see this,
given an identity neighbourhood $U$, choose an identity neighbourhood
$W$ such that
\[
\varphi^{r-1}(W)\varphi^{r-2}(W)\cdots\varphi(W)W
\subseteq U
\]
for every $1\leq r\leq m$. If $x_0,\ldots,x_k$ is a finite
$W$-pseudo-orbit and
$e_i:=\varphi(x_i)^{-1}x_{i+1}$, then
\[
\varphi^j(x_0)^{-1}x_j
=
\varphi^{j-1}(e_0)\varphi^{j-2}(e_1)\cdots e_{j-1}.
\]
Since $\varphi^m$ is trivial, the product on the right has at most
$m$ nontrivial factors and belongs to $U$. Thus $x_0$ shadows the
pseudo-orbit.

On the other hand, $\widetilde\beta_1$ is an automorphism of the
finite-dimensional compact connected abelian group
$\widehat{V_1}$. By Aoki's theorem \cite[Theorem 2]{AokiSolenoid},
$\widetilde\beta_1$ has shadowing if and only if
$T_{\mathbb Q}|_{V_1}$ is hyperbolic.

Since finite products preserve shadowing, and shadowing passes to
factors, we obtain
\[
\widetilde\beta\text{ has shadowing}
\quad\Longleftrightarrow\quad
\widetilde\beta_1\text{ has shadowing}.
\]
The only eigenvalue of $T_{\mathbb Q}|_{V_0}$ is $0$. Hence
\[
\operatorname{Spec}(T_{\mathbb Q})\cap\mathbb S^1
=
\operatorname{Spec}(T_{\mathbb Q}|_{V_1})\cap\mathbb S^1.
\]
Combining the preceding equivalences proves that $\beta$ has
shadowing if and only if $T_{\mathbb Q}$ is hyperbolic.
\end{proof}

\begin{remark}
Aoki proved Proposition~\ref{prop:finite-dim} for
automorphisms in \cite[Theorem~2]{AokiSolenoid}. In that case,
$T_{\mathbb Q}$ is invertible, so the nilpotent part in the Fitting
decomposition is trivial. The proposition above extends Aoki's result
to arbitrary endomorphisms; the only additional spectral possibility
is the eigenvalue $0$, which does not obstruct shadowing.
\end{remark}

\begin{theorem}\label{thm:abelian-endo-shad}
Let $A$ be a connected compact abelian group and let
$\beta:A\to A$ be an endomorphism. Then the following conditions are
equivalent:
\begin{enumerate}
\item $\beta$ has shadowing;
\item $T_{\mathbb Q}|_W$ is hyperbolic for every finite-dimensional
$T_{\mathbb Q}$-invariant subspace $W\leq V$.
\end{enumerate}
\end{theorem}

\begin{proof}
Assume first that $\beta$ has shadowing, and let $W\leq V$ be a
finite-dimensional $T_{\mathbb Q}$-invariant subspace. Put
$\Delta:=W\cap\Gamma$. Then $\Delta$ is a $T$-invariant subgroup of
$\Gamma$. Since $W$ is a finite-dimensional rational subspace of
$\Gamma\otimes_{\mathbb Z}\mathbb Q$, the subgroup
$\Delta=W\cap\Gamma$ has finite rank and spans $W$ over $\mathbb Q$.
Thus
$
\Delta\otimes_{\mathbb Z}\mathbb Q=W.
$

The inclusion $\Delta\hookrightarrow\Gamma$ dualizes to an equivariant
surjective homomorphism $A\to\widehat\Delta$. Hence the endomorphism
induced by $\beta$ on the finite-dimensional connected compact abelian group
$\widehat\Delta$ has shadowing. By Proposition \ref{prop:finite-dim}, $T_{\mathbb Q}|_W$ has no eigenvalue on the unit circle.

Conversely, assume that every finite-dimensional
$T_{\mathbb Q}$-invariant subspace of $V$ is hyperbolic. Regard $V$ as
a module over $\mathbb Q[u]$ by putting $u\cdot v=T_{\mathbb Q}v$.
Define its algebraic part by
\begin{equation*}
V_{\mathrm{alg}}
:=
\{v\in V:p(T_{\mathbb Q})v=0
\text{ for some }0\neq p\in\mathbb Q[u]\},
\end{equation*}
and put $\Gamma_{\mathrm{alg}}:=\Gamma\cap V_{\mathrm{alg}}$.

We have
\[
\Gamma_{\mathrm{alg}}\otimes_{\mathbb Z}\mathbb Q
=
V_{\mathrm{alg}}.
\]
Indeed, let $v\in V_{\mathrm{alg}}$. Choose a nonzero integer $m$
such that $mv\in\Gamma$. If
$p(T_{\mathbb Q})v=0$ for some nonzero $p\in\mathbb Q[u]$, then,
after clearing denominators,
$q(T)(mv)=0$ for some nonzero $q\in\mathbb Z[u]$. Hence
$mv\in\Gamma_{\mathrm{alg}}$, and therefore
$v\in\Gamma_{\mathrm{alg}}\otimes_{\mathbb Z}\mathbb Q$. The other inclusion is obvious.

We now consider the system dual to $\Gamma_{\mathrm{alg}}$. Let $\mathcal F$ be the family of all finite subsets of $\Gamma_{\mathrm{alg}}$ and for each $F\in \mathcal F$, let $W_F$ be the $\mathbb Q[u]$-submodule of
$V_{\mathrm{alg}}$ generated by $F$, and
put $\Delta_F:=W_F\cap\Gamma_{\mathrm{alg}}$.

Since every generator of $W_F$ is annihilated by a nonzero polynomial,
$W_F$ is finite-dimensional over $\mathbb Q$. The induced
endomorphism on $\widehat{\Delta_F}$ therefore has shadowing by the
finite-dimensional solenoidal criterion Proposition \ref{prop:finite-dim} and the hypothesis on the
spectrum.

Order $\mathcal F$ by inclusion. If $F\subseteq F'$, then
$\Delta_F\subseteq\Delta_{F'}$, and the inclusion dualizes to a
surjective equivariant homomorphism
$
\widehat{\Delta_{F'}}
\longrightarrow
\widehat{\Delta_F}.
$
Since
$
\Gamma_{\mathrm{alg}}
=
\bigcup_{F\in\mathcal F}\Delta_F,
$
Pontryagin duality gives
$
\widehat{\Gamma_{\mathrm{alg}}}
\cong
\varprojlim_{F\in\mathcal F}\widehat{\Delta_F}.
$
All bonding maps are surjective. Hence
Theorem~\ref{thm:inverse-limit-shadowing} implies that the induced
endomorphism on $\widehat{\Gamma_{\mathrm{alg}}}$ has shadowing.

Now put $\Lambda:=\Gamma/\Gamma_{\mathrm{alg}}$. Its rationalization
is naturally isomorphic to $V/V_{\mathrm{alg}}$, which is torsion-free
as a $\mathbb Q[u]$-module. Indeed, if
$p(T_{\mathbb Q})(v+V_{\mathrm{alg}})=0$, then
$p(T_{\mathbb Q})v\in V_{\mathrm{alg}}$, and hence
$q(T_{\mathbb Q})\big(p(T_{\mathbb Q})v\big)=0$ for some nonzero
$q\in\mathbb Q[u]$. Thus $v\in V_{\mathrm{alg}}$.

Again, we let $\mathcal E$ be the set of finite subsets of $\Lambda$ and for each $E\in \mathcal E$, let $M_E$ be the $\mathbb Q[u]$-submodule of
$\Lambda\otimes\mathbb Q$ generated by $E$, and put
$\Lambda_E:=M_E\cap\Lambda$. Since $\mathbb Q[u]$ is a principal ideal
domain and $M_E$ is finitely generated and torsion-free, there is an
integer $r_E$ such that
\[
M_E\cong\mathbb Q[u]^{r_E}.
\]

The dual of multiplication by $u$ on $\mathbb Q[u]^{r_E}$ is a
one-sided full shift on
$(\widehat{\mathbb Q}^{\,r_E})^{\mathbb N}$, and hence has shadowing.
Indeed, as an abelian group, one has
\[
\mathbb Q[u]^{r_E}
=
\bigoplus_{j\geq 0}\mathbb Q^{r_E}u^j.
\]
Therefore, by Pontryagin duality,
\[
\widehat{\mathbb Q[u]^{r_E}}
\cong
\prod_{j\geq 0}\widehat{\mathbb Q^{r_E}}
\cong
\bigl(\widehat{\mathbb Q}^{\,r_E}\bigr)^{\mathbb N}.
\]
More explicitly, every element of $\mathbb Q[u]^{r_E}$ can be written
uniquely in the form
\[
p(u)=\sum_{j\geq 0}a_j u^j,
\qquad a_j\in\mathbb Q^{r_E},
\]
where all but finitely many $a_j$ are zero. Under the above
identification, a point
$x=(x_0,x_1,\ldots)\in
(\widehat{\mathbb Q}^{\,r_E})^{\mathbb N}$
corresponds to the character $\chi_x$ defined by
\[
\chi_x(p)
=
\prod_{j\geq 0}x_j(a_j)\in \mathbb T.
\]

Let $m_u:\mathbb Q[u]^{r_E}\to\mathbb Q[u]^{r_E}$ denote
multiplication by $u$. Its dual map satisfies
\[
\widehat{m_u}(\chi_x)(p)
=
\chi_x(up).
\]
Since
$
up(u)=\sum_{j\geq 0}a_j u^{j+1},
$
we obtain
$
\chi_x(up)
=
\prod_{j\geq 0}x_{j+1}(a_j).
$
Thus, under the identification above,
\[
\widehat{m_u}(x_0,x_1,x_2,\ldots)
=
(x_1,x_2,x_3,\ldots).
\]
Hence the dual of multiplication by $u$ on
$\mathbb Q[u]^{r_E}$ is precisely the one-sided full shift on
$\bigl(\widehat{\mathbb Q}^{\,r_E}\bigr)^{\mathbb N}$.

The inclusion $\Lambda_E\hookrightarrow M_E$ dualizes to an
equivariant surjective homomorphism
$\widehat{M_E}\to\widehat{\Lambda_E}$. It follows from Lemma \ref{lem:quotient-shadowing} that the induced
endomorphism on $\widehat{\Lambda_E}$ has shadowing.

Order $\mathcal E$ by inclusion. If $E\subseteq E'$, then
$M_E\subseteq M_{E'}$ and hence
$\Lambda_E\subseteq\Lambda_{E'}$. Moreover,
$
\Lambda=\bigcup_{E\in\mathcal E}\Lambda_E.
$
Therefore
$
\widehat\Lambda
\cong
\varprojlim_{E\in\mathcal E}\widehat{\Lambda_E},
$
where all bonding maps are surjective and equivariant. Theorem~
\ref{thm:inverse-limit-shadowing} now implies that the induced
endomorphism on $\widehat\Lambda$ has shadowing.

Finally, the exact sequence
\[
0\longrightarrow\Gamma_{\mathrm{alg}}
\longrightarrow\Gamma
\longrightarrow\Lambda
\longrightarrow 0
\]
dualizes to an equivariant exact sequence
\[
1\longrightarrow\widehat\Lambda
\longrightarrow A
\longrightarrow\widehat{\Gamma_{\mathrm{alg}}}
\longrightarrow 1.
\]
Both the restriction to $\widehat\Lambda$ and the induced
endomorphism on $\widehat{\Gamma_{\mathrm{alg}}}$ have shadowing.
Lemma~\ref{lem:extension-shadowing} therefore implies that $\beta$ has
shadowing.
\end{proof}

\subsection{Surjective endomorphisms of semisimple compact groups}

We first describe the form of a surjective endomorphism on direct products of simple connected compact Lie groups.

\begin{lemma}\label{lem:semisimple-index-map}
Let
$
P:=\prod_{i\in I}S_i,
$
where every $S_i$ is a centre-free simple connected compact Lie group,
and let $\beta:P\to P$ be a surjective endomorphism. Then there exist
an injective map $\tau:I\to I$ and isomorphisms
$\theta_i:S_{\tau(i)}\to S_i$ such that
\begin{equation}\label{eq:coordinate-endo}
p_i\circ\beta=\theta_i\circ p_{\tau(i)}
\end{equation}
for every $i\in I$, where $p_i:P\to S_i$ denotes the coordinate
projection.
\end{lemma}

\begin{proof}
Fix $i\in I$ and put $\rho_i:=p_i\circ\beta$. Since $S_i$ is a Lie
group, it has the no-small-subgroups property. It follows from the
continuity of $\rho_i$ that $\rho_i$ depends on only finitely many
coordinates.

For each $j\in I$, the image under $\rho_i$ of the $j$-th factor is a
connected normal subgroup of $S_i$. Since $S_i$ is simple, this image
is either trivial or equal to $S_i$. Moreover, the images of two
distinct factors commute. Since $\rho_i$ is surjective and $S_i$ is
nonabelian, exactly one factor has nontrivial image.

Denote this factor by $S_{\tau(i)}$. The restriction of $\rho_i$ to
$S_{\tau(i)}$ is a surjective homomorphism onto $S_i$. Its kernel is
a closed normal subgroup of the centre-free simple group
$S_{\tau(i)}$, and hence it is trivial. Thus this restriction is an
isomorphism, which we denote by $\theta_i$.

It remains to show that $\tau$ is injective. Suppose that
$\tau(i)=\tau(j)$ for distinct $i,j\in I$. Then the projection of
$\beta(P)$ onto $S_i\times S_j$ is contained in the graph of an
isomorphism between $S_i$ and $S_j$, and hence cannot be the whole
product. This contradicts the surjectivity of $\beta$.
\end{proof}

For an injective map $\tau:I\to I$, each component of its functional
graph is one of the following:
\begin{enumerate}
\item a finite cycle;
\item a copy of $\mathbb Z$;
\item a copy of $\mathbb N$.
\end{enumerate}
The second type gives a two-sided shift, while the third type gives a
one-sided shift.

We shall also use the following elementary observation, which is folklore. For the reader's convenience, we provide a proof here.

\begin{lemma}\label{lem:iso-no-shad}
Let $X$ be a nontrivial compact connected metric space and let
$f:X\to X$ be an isometry. Then $f$ does not have finite shadowing.
\end{lemma}

\begin{proof}
	Suppose, towards a contradiction, that $f$ has finite shadowing.
	Choose $a,b\in X$ and $\varepsilon>0$ such that
	$d(a,b)>2\varepsilon$. Let $\delta>0$ be such that every finite
	$\delta$-pseudo-orbit is $\varepsilon$-shadowed.

Since $X$ is connected, it is $\delta$-chain connected. Choose
$a=c_0,c_1,\ldots,c_n=b$ such that
$d(c_i,c_{i+1})<\delta$ for every $i<n$, and put
$x_i:=f^i(c_i)$. Then
\[
d(f(x_i),x_{i+1})=d(c_i,c_{i+1})<\delta,
\]
so $x_0,\ldots,x_n$ is a finite $\delta$-pseudo-orbit.

If it were $\varepsilon$-shadowed by $x\in X$, then
$d(f^i(x),f^i(c_i))<\varepsilon$ for every $i$. Since $f$ is an
isometry, this gives $d(x,c_i)<\varepsilon$ for every $i$, contrary
to $d(a,b)>2\varepsilon$.
\end{proof}

\begin{theorem}\label{thm:semisimple-sur}
Let $S$ be a semisimple connected compact group and let
$\beta:S\to S$ be a surjective endomorphism. Let
$\tau:I\to I$ be the injective map induced by $\beta$ on the simple
factors of $S/Z(S)$ as in
Lemma~\ref{lem:semisimple-index-map}. Then the following conditions
are equivalent:
\begin{enumerate}
\item $\beta$ has shadowing;
\item $\tau$ has no periodic point;
\item every component of the functional graph of $\tau$ is a copy of
$\mathbb N$ or $\mathbb Z$.
\end{enumerate}
\end{theorem}

\begin{proof}
The equivalence of (2) and (3) follows from the injectivity of $\tau$.

Assume that $\beta$ has shadowing. By
Lemma~\ref{lem:quotient-shadowing}, the induced endomorphism
$\bar\beta$ on $S/Z(S)$ has shadowing.

Suppose that $\tau$ has a periodic orbit $C\subseteq I$. Then $C$ is
finite, and the projection onto
\[
S_C:=\prod_{i\in C}S_i
\]
is an equivariant factor map. The endomorphism induced by
$\bar\beta$ on $S_C$ is an automorphism of a nontrivial compact
connected semisimple Lie group.

Such an automorphism preserves the Riemannian metric induced by the
negative Killing form. It is therefore an isometry, contradicting
Lemma~\ref{lem:iso-no-shad}. Thus $\tau$ has no
periodic point.

Conversely, suppose that $\tau$ has no periodic point. By recursively
changing the coordinates along each component of the functional graph
of $\tau$, the isomorphisms $\theta_i$ in
\eqref{eq:coordinate-endo} can be removed. Consequently,
$(S/Z(S),\bar\beta)$ is conjugate to a product of systems of the form
\[
(H^{\mathbb N},\sigma_+)
\qquad\text{and}\qquad
(H^{\mathbb Z},\sigma),
\]
where $H$ is a centre-free simple connected compact Lie group,
$\sigma_+$ is the one-sided shift and $\sigma$ is the two-sided shift.

Both one-sided and two-sided full shifts have shadowing. Moreover, arbitrary products of systems with shadowing also have shadowing. 
Indeed, finite products preserve shadowing, while an arbitrary product is the inverse limit of its finite subproducts, with all bonding maps surjective. 
The conclusion therefore follows from the inverse-limit theorem for shadowing.

Finally, $Z(S)$ is compact and zero-dimensional, so
$\beta|_{Z(S)}$ has shadowing. Applying
Lemma~\ref{lem:extension-shadowing} to
\[
1\longrightarrow Z(S)\longrightarrow S
\longrightarrow S/Z(S)\longrightarrow 1
\]
shows that $\beta$ has shadowing.
\end{proof}

\begin{remark}\label{rem:semisimple-shifts}
The $\mathbb N$-components of $\tau$ correspond precisely to
one-sided shift factors, whereas the $\mathbb Z$-components correspond
to bilateral shift factors. Thus the new phenomenon for surjective
endomorphisms, in comparison with automorphisms, is the occurrence of
one-sided shifts.
\end{remark}

\subsection{The general connected case}

Let $K$ be a connected compact group. Recall from Section~\ref{pre} that
\[
K=AK',
\qquad
A:=Z(K)_0,
\]
and $D:=A\cap K'$ is a zero-dimensional central subgroup. Moreover,
$K'$ is connected and semisimple.

Let $\beta:K\to K$ be a surjective endomorphism. Then
$\beta(A)\subseteq A$ and $\beta(K')=K'$. Indeed, the first inclusion
follows because the image of the connected central group $A$ is
contained in $Z(K)_0$, while
\[
\beta(K')
=
[\beta(K),\beta(K)]
=
[K,K]
=
K'.
\]

\begin{proposition}\label{prop:central-semisimple-reduction}
Let $K$ be a compact connected group and let
$\beta:K\to K$ be a surjective endomorphism. Put
$A:=Z(K)_0$ and $S:=K'$. Then $\beta$ has shadowing if and only if
both $\beta|_A$ and $\beta|_S$ have shadowing.
\end{proposition}

\begin{proof}
Suppose first that $\beta|_A$ and $\beta|_S$ have shadowing. Then the
product endomorphism on $A\times S$ has shadowing. The multiplication
map
\[
m:A\times S\to K,\qquad m(a,s)=as,
\]
is a continuous surjective homomorphism intertwining this product
endomorphism with $\beta$. Hence $\beta$ has shadowing by
Lemma~\ref{lem:quotient-shadowing}.

Conversely, suppose that $\beta$ has shadowing. Since
$\beta(A)\subseteq A$ and $\beta(S)\subseteq S$, we have
$\beta(D)\subseteq D$, where $D:=A\cap S$. Hence $\beta$ induces an
endomorphism $\bar\beta$ on $K/D$. By
Lemma~\ref{lem:quotient-shadowing}, $\bar\beta$ has shadowing.

Since $K=AS$, the subgroups $A/D$ and $S/D$ commute, have trivial
intersection, and generate $K/D$. Therefore the map
\[
(A/D)\times(S/D)\longrightarrow K/D,\qquad
(aD,sD)\longmapsto asD,
\]
is a topological isomorphism. Under this identification,
$\bar\beta$ is the product endomorphism
\[
\bar\beta_A\times\bar\beta_S,
\]
where $\bar\beta_A$ and $\bar\beta_S$ are induced by $\beta|_A$ and
$\beta|_S$, respectively.

The coordinate projections are equivariant surjective homomorphisms.
It follows from Lemma~\ref{lem:quotient-shadowing} that both
$\bar\beta_A$ and $\bar\beta_S$ have shadowing. Since $D$ is
zero-dimensional, $\beta|_D$ has shadowing. Applying
Lemma~\ref{lem:extension-shadowing} to the exact sequences
\[
1\longrightarrow D\longrightarrow A\longrightarrow A/D
\longrightarrow 1
\]
and
\[
1\longrightarrow D\longrightarrow S\longrightarrow S/D
\longrightarrow 1,
\]
we conclude that both $\beta|_A$ and $\beta|_S$ have shadowing.
\end{proof}

We can now state the main result.

\begin{theorem}\label{thm:connected-endo-shad}
Let $G$ be a compact connected group and let
$\alpha:G\to G$ be an endomorphism. Put
\[
K:=G_\infty=\bigcap_{n\geq 0}\alpha^n(G),
\qquad
\beta:=\alpha|_K,
\]
and let
\[
A:=Z(K)_0,
\qquad
S:=K'.
\]

Let $T_{\mathbb Q}$ be the rational linear endomorphism induced by
$\widehat{\beta|_A}$ on
\[
\widehat A\otimes_{\mathbb Z}\mathbb Q.
\]
Let $\tau:I\to I$ be the injective map induced by $\beta|_S$ on the
simple Lie factors of
\[
S/Z(S)\cong\prod_{i\in I}S_i.
\]
Then the following conditions are equivalent:
\begin{enumerate}
\item $\alpha$ has shadowing;
\item
\begin{enumerate}
\item for every finite-dimensional $T_{\mathbb Q}$-invariant subspace
$W\leq\widehat A\otimes_{\mathbb Z}\mathbb Q$, one has
\[
\operatorname{Spec}(T_{\mathbb Q}|_W)\cap\mathbb S^1
=
\varnothing;
\]
\item the map $\tau$ has no periodic point.
\end{enumerate}
\end{enumerate}
\end{theorem}

\begin{proof}The subgroup $K=G_\infty$ is connected, and
	$\beta=\alpha|_K$ is surjective by
	Lemma~\ref{lem:stable-image-surjective}.
	
By \eqref{eq:stable-image-equiv}, $\alpha$ has shadowing if and
only if $\beta$ has shadowing. By
Proposition~\ref{prop:central-semisimple-reduction}, this is equivalent
to both $\beta|_A$ and $\beta|_S$ having shadowing.

Theorem~\ref{thm:abelian-endo-shad} shows that $\beta|_A$ has
shadowing if and only if condition~(a) holds. Since $\beta|_S$ is
surjective, Theorem~\ref{thm:semisimple-sur} shows
that $\beta|_S$ has shadowing if and only if condition~(b) holds.
\end{proof}

\begin{remark}
Let $\overline\beta_S$ be the endomorphism of $S/Z(S)$ induced by
$\beta|_S$. The condition that $\tau$ has no periodic point admits an
equivalent dynamical formulation. Indeed, since $\tau:I\to I$ is
injective, every connected component of its directed graph is either
a one-sided chain, indexed by $\mathbb N$, or a two-sided chain,
indexed by $\mathbb Z$, provided that $\tau$ has no periodic point.
After suitable coordinatewise identifications of the simple Lie
factors along each component, the restriction of
$\overline\beta_S$ to that component becomes, respectively, a
one-sided or a two-sided full shift. Consequently,
\[
\tau\text{ has no periodic point}
\]
if and only if $\overline\beta_S$ is topologically conjugate to a
product of one-sided and two-sided full shifts on simple connected compact Lie groups.

Thus condition~{\rm(2)(b)} in
Theorem~\ref{thm:connected-endo-shad} may equivalently be stated as
requiring the endomorphism induced by $\beta|_S$ on $S/Z(S)$ to admit
such a shift decomposition.
\end{remark}

\subsection{The automorphism case}

For an automorphism, the stable image is the whole group. Moreover,
the induced linear map on the abelian dual is invertible, and the
index map on the simple semisimple factors is a permutation. We
therefore obtain the following consequence.

\begin{corollary}\label{cor:connected-aut-shadowing}
Let $G$ be a compact connected group and let
$\alpha\in\operatorname{Aut}(G)$. Put
\[
A:=Z(G)_0,
\qquad
S:=G'.
\]
Let $T_{\mathbb Q}$ be the automorphism induced by
$\widehat{\alpha|_A}$ on
$\widehat A\otimes_{\mathbb Z}\mathbb Q$, and let
$\tau:I\to I$ be the permutation induced by $\alpha|_S$ on the simple
Lie factors of
\[
S/Z(S)\cong\prod_{i\in I}S_i.
\]
Then the following conditions are equivalent:
\begin{enumerate}
\item $\alpha$ has shadowing;
\item $\alpha$ has two-sided shadowing;
\item
\begin{enumerate}
\item every finite-dimensional $T_{\mathbb Q}$-invariant subspace is
hyperbolic;
\item every orbit of $\tau$ is infinite.
\end{enumerate}
\end{enumerate}
\end{corollary}

\begin{proof}
For a homeomorphism of a compact space, shadowing and two-sided
shadowing are equivalent. The remaining equivalence follows from
Theorem~\ref{thm:connected-endo-shad}.

Indeed, since $\alpha$ is an automorphism, $T_{\mathbb Q}$ is
invertible, so zero cannot be an eigenvalue. Moreover, $\tau$ is
bijective. For a permutation, having no periodic point is equivalent
to every orbit being infinite. Each such orbit is indexed by
$\mathbb Z$ and gives a bilateral shift component.
\end{proof}

\begin{remark}
If $G$ is a nontrivial compact connected semisimple Lie group, then
the index set $I$ is finite. Every permutation of $I$ has a finite
orbit, and hence no automorphism of $G$ has shadowing.
\end{remark}

\subsection{Proof of the metrizable case of
	Lemma~\ref{lem:aut-component}}\label{subsec}

In what follows, we assume that $G$ is a compact metrizable group and
that $\gamma\in\Aut(G)$ has shadowing. We shall use the following
lemmas.

\begin{lemma}\label{lem:abelian-comp}
	\cite[Lemma~23]{AD}
	Let $K$ be a compact metrizable group whose identity component $K_0$
	is finite-dimensional and abelian. For every $\tau\in\Aut(K)$, there
	exists a closed normal totally disconnected subgroup $L$ of $K$ such
	that
	\[
	\tau(L)=L
	\qquad\text{and}\qquad
	K_0L\text{ is open in }K.
	\]
\end{lemma}

\begin{lemma}\label{lem:AD-semisimple}
	\cite[Lemma~29]{AD}
	Let $K$ be a compact metrizable group such that $K_0$ is centre-free,
	and let $\tau\in\Aut(K)$. If $(K,\tau)$ has the shadowing property, then
	$K_0$ contains no nontrivial closed Lie subgroup $U$ such that
	\[
	\tau(U)=U
	\qquad\text{and}\qquad
	U\trianglelefteq K_0.
	\]
\end{lemma}

Put
\[
C:=G_0,
\qquad
\beta:=\gamma|_C,
\qquad
A:=Z(C)_0,
\qquad
S:=C'.
\]
We verify the two conditions in
Theorem~\ref{thm:connected-endo-shad} for $\beta$.

We first make a standard reduction. The subgroup $Z(S)$ is totally
disconnected, characteristic in $S$, and hence normal and
$\gamma$-invariant in $G$. Thus $\gamma|_{Z(S)}$ has shadowing. Since
\[
C/Z(S)\cong A/D\times S/Z(S),
\qquad D:=A\cap S,
\]
Lemma~\ref{lem:extension-shadowing} shows that it is enough to prove
that the automorphism induced by $\beta$ on $C/Z(S)$ has shadowing.
After replacing $G$ by $G/Z(S)$ and relabelling the corresponding
factors, we may therefore assume that
\[
C=A\times S,
\]
where $A$ is connected compact abelian and $S$ is a product of
centre-free simple connected compact Lie groups.

We also record a simple observation. If $H$ is an open
$\gamma$-invariant subgroup of $G$, then $\gamma|_H$ has shadowing.
Indeed, let $U$ be an identity neighbourhood in $H$. Since $H$ is
open, $U$ is also an identity neighbourhood in $G$. Choose an identity
neighbourhood $V\subseteq H$ such that every $V$-pseudo-orbit in $G$
is $U$-shadowed. If a $V$-pseudo-orbit contained in $H$ is shadowed by
$x\in G$, then
\[
x^{-1}x_0\in U\subseteq H.
\]
Since $x_0\in H$, it follows that $x\in H$. Hence
$\gamma|_H$ has shadowing.

We now consider the abelian part. Put
\[
\Gamma:=\widehat A,
\qquad
T:=\widehat{\gamma|_A},
\]
and let $T_{\mathbb Q}$ denote the induced automorphism of
$\Gamma\otimes_{\mathbb Z}\mathbb Q$.

Suppose, towards a contradiction, that there is a finite-dimensional
$T_{\mathbb Q}$-invariant subspace
\[
W_0\leq\Gamma\otimes_{\mathbb Z}\mathbb Q
\]
such that $T_{\mathbb Q}|_{W_0}$ is not hyperbolic.

Since $W_0\cap\Gamma$ spans $W_0$ over $\mathbb Q$, choose a
$\mathbb Q$-basis
\[
\chi_1,\ldots,\chi_r\in W_0\cap\Gamma.
\]
For each $j$, let
\[
W_{\chi_j}
=
\operatorname{span}_{\mathbb Q}
\{T^n\chi_j:n\in\mathbb Z\}.
\]
Then
\[
W_0=W_{\chi_1}+\cdots+W_{\chi_r}.
\]
Consequently, the minimal polynomial of $T_{\mathbb Q}|_{W_0}$ is
the least common multiple of the minimal polynomials of the
restrictions $T_{\mathbb Q}|_{W_{\chi_j}}$. Since
$T_{\mathbb Q}|_{W_0}$ is not hyperbolic, for some $j$ the minimal
polynomial of $T_{\mathbb Q}|_{W_{\chi_j}}$ has a root on the unit
circle. Fix such a $j$ and put $\chi:=\chi_j$.

Conjugation by $G$ on $A$ induces an action of $G$ on $\Gamma$,
defined by
\[
(g\cdot\eta)(a)
=
\eta(g^{-1}ag),
\qquad
g\in G,\ \eta\in\Gamma,\ a\in A.
\]
We extend this action $\mathbb Q$-linearly to
$\Gamma\otimes_{\mathbb Z}\mathbb Q$.

For every $\eta\in\Gamma$, the orbit map
\[
G\longrightarrow\Gamma,\qquad g\longmapsto g\cdot\eta,
\]
is continuous. Since $G$ is compact and $\Gamma$ is discrete, every
$G$-orbit in $\Gamma$ is finite. Moreover,
\[
T(g\cdot\eta)
=
\gamma^{-1}(g)\cdot T\eta.
\]

Put
\[
W
:=
\operatorname{span}_{\mathbb Q}
\{g\cdot\eta:g\in G,\ \eta\in W_\chi\}.
\]
The covariance relation shows that $W$ is invariant under
$T_{\mathbb Q}$, and it is clearly $G$-invariant.

Since $W_\chi$ is generated by the elements $T^n\chi$, we may choose
a basis
$
\eta_1,\ldots,\eta_m
$
of $W_\chi$ with each $\eta_j\in\Gamma$. Each $G$-orbit
$G\cdot\eta_j$ is finite. Hence
\[
W
=
\operatorname{span}_{\mathbb Q}
\big(\bigcup_{j=1}^mG\cdot\eta_j\big)
\]
is finite-dimensional.

 Since $W_\chi$ is a $T_{\mathbb Q}$-invariant
subspace of $W$, the restriction $T_{\mathbb Q}|_W$ is still not
hyperbolic.

Put
$
\Delta:=W\cap\Gamma.
$
Then $\Delta$ is a finite-rank, $G$-invariant and $T$-invariant
subgroup of $\Gamma$, and
$
\Delta\otimes_{\mathbb Z}\mathbb Q=W.
$
Let
\[
B:=\Delta^\perp
=
\{a\in A:\delta(a)=1\text{ for every }\delta\in\Delta\}.
\]
The $G$-invariance of $\Delta$ implies that $B$ is normal in $G$, and
the $T$-invariance of $\Delta$ implies that $\gamma(B)=B$. Moreover,
$
\widehat{A/B}\cong\Delta,
$
so $A/B$ is a finite-dimensional compact connected abelian group. Since $S=C'$ is characteristic in $C$ and $C=G_0$ is characteristic
in $G$, the subgroup $S$ is closed, normal and $\gamma$-invariant in
$G$. Since $C=A\times S$ and $B\leq A$, the subgroup $B\times S$ is
closed, normal and $\gamma$-invariant in $G$.

Consider the quotient
\[
K:=G/(B\times S),
\]
and let $\gamma_K\in\Aut(K)$ be the automorphism induced by $\gamma$. Since
$C=A\times S$, we have
\[
K_0\cong A/B.
\]
In particular, $K_0$ is finite-dimensional and abelian. By
Lemma~\ref{lem:quotient-shadowing}, $\gamma_K$ has shadowing.

Applying Lemma~\ref{lem:abelian-comp}, choose a closed normal totally
disconnected subgroup $L$ of $K$ such that
\[
\gamma_K(L)=L
\qquad\text{and}\qquad
K_0L\text{ is open in }K.
\]
The restriction
$
\lambda:=\gamma_K|_{K_0L}
$
has shadowing by the preceding observation. Since $L$ is normal and
$\lambda$-invariant, the induced automorphism
$
\overline\lambda
$
on
\[
(K_0L)/L
\cong
K_0/(K_0\cap L)
\]
also has shadowing.

Put
$
Q:=K_0\cap L.
$
Then $Q$ is totally disconnected. Under the identification
$
\widehat{K_0}\cong\Delta,
$
we have
\[
\widehat{K_0/Q}
\cong
Q^\perp
:=
\{\delta\in\Delta:\delta|_Q=1\}.
\]
The exact sequence
\[
0\longrightarrow Q^\perp
\longrightarrow \Delta
\longrightarrow \widehat Q
\longrightarrow 0
\]
shows that $\Delta/Q^\perp$ is torsion, because $Q$ is totally
disconnected and hence $\widehat Q$ is torsion. It follows that
\[
Q^\perp\otimes_{\mathbb Z}\mathbb Q
=
\Delta\otimes_{\mathbb Z}\mathbb Q
=
W.
\]

The dual automorphism of $\overline\lambda$ is the restriction of
$T$ to $Q^\perp$. Therefore its rational extension is precisely
$
T_{\mathbb Q}|_W.
$
Since $\overline\lambda$ has shadowing,
Proposition~\ref{prop:finite-dim} implies that
$T_{\mathbb Q}|_W$ is hyperbolic. This contradicts the construction
of $W$.

Hence $T_{\mathbb Q}|_W$ is hyperbolic for every finite-dimensional
$T_{\mathbb Q}$-invariant subspace
$W\leq\Gamma\otimes_{\mathbb Z}\mathbb Q$.

We now consider the semisimple part. By the reduction made at the
beginning of the proof, we may assume that
$S=\prod_{i\in I}S_i$, where each $S_i$ is a centre-free simple connected compact Lie group.

Let $p_i:S\to S_i$ denote the coordinate projection. The
automorphism $\beta|_S$ induces a permutation $\tau:I\to I$ and
isomorphisms $\theta_i:S_{\tau(i)}\to S_i$ such that
\[
p_i\circ\beta|_S=\theta_i\circ p_{\tau(i)}
\]
for every $i\in I$. Equivalently,
$\beta(S_{\tau(i)})=S_i$.

We claim that $\tau$ has no periodic point. Suppose otherwise, and let
$J\subseteq I$ be a finite orbit of $\tau$. Then the subproduct
\[
U:=\prod_{j\in J}S_j
\]
is a nontrivial compact connected semisimple Lie group. Since
$\tau(J)=J$, the preceding coordinate description implies that
$\beta(U)=U$. Moreover, $U$ is a closed normal subgroup of $S$.

Consider the quotient group $\overline G:=G/A$, and let
$\overline\gamma$ be the automorphism induced by $\gamma$.  Lemma~
\ref{lem:quotient-shadowing} implies that $\overline\gamma$ has
shadowing.

By our initial reduction, $C=A\times S$. Therefore
$\overline G_0=C/A\cong S$, which is centre-free. Under this
identification, $U$ is a nontrivial closed Lie subgroup of
$\overline G_0$, normal in $\overline G_0$, and invariant under
$\overline\gamma$. This contradicts Lemma~\ref{lem:AD-semisimple}.

Thus $\tau$ has no periodic point. Together with the conclusion
obtained for the abelian part, both conditions in
Theorem~\ref{thm:connected-endo-shad} are satisfied. Consequently,
$\beta=\gamma|_{G_0}$ has shadowing. This completes the proof of the
metrizable case of Lemma~\ref{lem:aut-component}.

\end{document}